\documentclass[11pt, reqno]{amsart}

\usepackage[utf8]{inputenc}
\usepackage[T2A]{fontenc}
\usepackage[ukrainian,russian,english]{babel}

\usepackage[english]{d-omega3}

\usepackage{amsmath,amsfonts,amssymb}
\usepackage{graphicx}
\usepackage{xcolor}
\usepackage[all]{xy}
\usepackage{tikz}
\usepackage{multirow}
\usepackage{verbatim}
\usepackage{url}
\usepackage{hyperref}
\usepackage{float}
\floatstyle{plaintop}
\restylefloat{table}

\newcommand{\Cb}{\mathbb{C}}

\begin{document}


\author{Ji-Young Ham}
\email{jiyoungham1@gmail.com}
\address{School of Liberal Arts, Seoul National University of Science and Technology,
232 Gongneung-ro, Nowon-gu, Seoul, korea 
01811 \\
Department of Architecture, Konkuk University, 
120 Neungdong-ro, Gwangjin-gu, Seoul, Korea
05029}

\author{Joongul Lee}
\email{jglee@hongik.ac.kr}
\address{Department of Mathematics Education, Hongik University, 
94 Wausan-ro, Mapo-gu, Seoul, Korea
04066}

\title[Chern-Simons invariants for double twist knot orbifolds]{Explicit formulae for Chern-Simons invariants of the hyperbolic $J(2n,-2m)$ knot orbifolds}

\abstract{english}{We calculate the Chern-Simons invariants of  the hyperbolic double twist knot orbifolds using the Schl\"{a}fli formula for the generalized Chern-Simons function on the family of cone-manifold structures of double twist knots.}

\keywords{Chern-Simons invariant, double twist knot, orbifold, Riley-Mednykh polynomial, orbifold covering}
\shortAuthorsList{J. Ham and J. Lee}
\msc{57M25,  57M27.}
\thanks{This work was supported by Basic Science Research Program through the National Research Foundation of Korea (NRF) funded by the Ministry of Education, Science and Technology (No. NRF-2018R1A2B6005847). The second author was supported by 2019 Hongik University Research Fund. The authors would like to thank Alexander Mednykh, Nathan Dunfield, and Darren Long.}

\maketitle

\section{Introduction}
Chern-Simons invariants of hyperbolic knot orbifolds are computed explicitly for a few infinite families in~\cite{HL,HL2,HLMR} using the ``Schl\"{a}fli formula''. 

In this paper, we present the explicit formulae for Chern-Simons invariants of the hyperbolic double twist knot orbifolds and we present them numerically for some of double twist knot orbifolds. A brief history of Chern-Simons invariant  can be found in~\cite{HL,HL2,HLMR}. A double twist knot is denoted by $C(2n,2m)$ according to Conway notation or by $J(2n,-2m)$ according to Hoste-Shanahan notation. Figure~\ref{sec:genOne} presents $C(2n,2m)$ for $m, n >0$. 

For a two-bridge hyperbolic link, there exists an angle $\alpha_0 \in [\frac{2\pi}{3},\pi)$ for each link $K$ such that the cone-manifold $K(\alpha)$ is hyperbolic for $\alpha \in (0, \alpha_0)$, Euclidean for $\alpha=\alpha_0$, and spherical for $\alpha \in (\alpha_0, \pi]$ \cite{P2,HLM1,K1,PW}. We will use the Chern-Simons invariant of the lens space $L\left(4nm+1,2n(2m-1)+1\right)$ calculated in~\cite{HLM2}.
The following theorem gives the Chern-Simons invariant formulae  for the hyperbolic $C(2n,2m)$ knots.
Let $S_k(v)$ be the 
\emph{Chebychev polynomials} defined by $S_0(v) = 1$, $S_1(v) = v$ and $S_k(v) = vS_{k-1}(v) - S_{k-2}(v)$ for all integers $k$.

\begin{theorem}\label{theorem:main}
Let $X_{2n}^{2m} \!  (\alpha)$  be the hyperbolic cone-manifold with underlying space $S^3$ and a singular set $C(2n,2m)$ of cone-angle $0 \leq \alpha < \alpha_0$. Let $k$ be a positive integer such that $k$-fold cyclic covering of 
$X_{2n}^{2m}(\frac{2 \pi}{k})$ is hyperbolic. Then the Chern-Simons invariant of $X_{2n}^{2m}(\frac{2 \pi}{k})$ (mod $\frac{1}{k}$ if $k$ is even or mod $\frac{1}{2k}$ if $k$ is odd) is given by the following formula:

{\small
\begin{align*}
&\text{\textnormal{cs}} \left(X_{2n}^{2m} \left(\frac{2 \pi}{k} \right)\right) 
 \equiv \frac{1}{2} \text{\textnormal{cs}}\left(L\left(4nm+1,2n(2m-1)+1\right) \right) \\
&+\frac{1}{4 \pi^2}\int_{\frac{2 \pi}{k}}^{\alpha_0} Im \left(2*\log \left(-\frac{M^2 (S_n(v)-S_{n-1}(v))-(S_{n-1}(v)-S_{n-2}(v))}{(S_n(v)-S_{n-1}(v))-M^2 (S_{n-1}(v)-S_{n-2}(v))}\right)\right) \: d\alpha \\
& +\frac{1}{4 \pi^2}\int_{\alpha_0}^{\pi}
 Im \left(\log \left(-\frac{M^2 (S_n(v_1)-S_{n-1}(v_1))-(S_{n-1}(v_1)-S_{n-2}(v_1))}{(S_n(v_1)-S_{n-1}(v_1))-M^2 (S_{n-1}(v_1)-S_{n-2}(v_1))}\right) \right) \: d\alpha\\
& +\frac{1}{4 \pi^2}\int_{\alpha_0}^{\pi}
 Im \left(\log \left(-\frac{M^2 (S_n(v_2)-S_{n-1}(v_2))-(S_{n-1}(v_2)-S_{n-2}(v_2))}{(S_n(v_2)-S_{n-1}(v_2))-M^2 (S_{n-1}(v_2)-S_{n-2}(v_2))}\right)\right) \: d\alpha,
\end{align*}
}
\noindent where 
for $M=e^{\frac{i \alpha}{2}}$, 
$x$ , $x_1$, and $x_2$ are zeroes of Riley-Mednykh polynomial $\phi_{2n}^{2m} (x,M)$ in \text{\textnormal{Theorem}}~\ref{thm:RMpolynomial}.
As $\alpha$ decreases to $\alpha_0$ both $x_1$ and $x_2$ approach a common value $x$. One of $x_1$ and $x_2$ comes from the component of $x$, and the other comes from the component of $\bar{x}$.
$v$ satisfies $\text{\textnormal{Im}} ((S_n(v)-S_{n-1}(v))\overline{(S_{n-1}(v)-S_{n-2}(v))}) \geq 0$~\cite[Lemma 3.9]{Tran3}.
\end{theorem}


\section{$C(2n,2m)$ knots} \label{sec:genOne}
\begin{figure} 
\begin{center}
\resizebox{3.5cm}{!}{\includegraphics[angle=0]{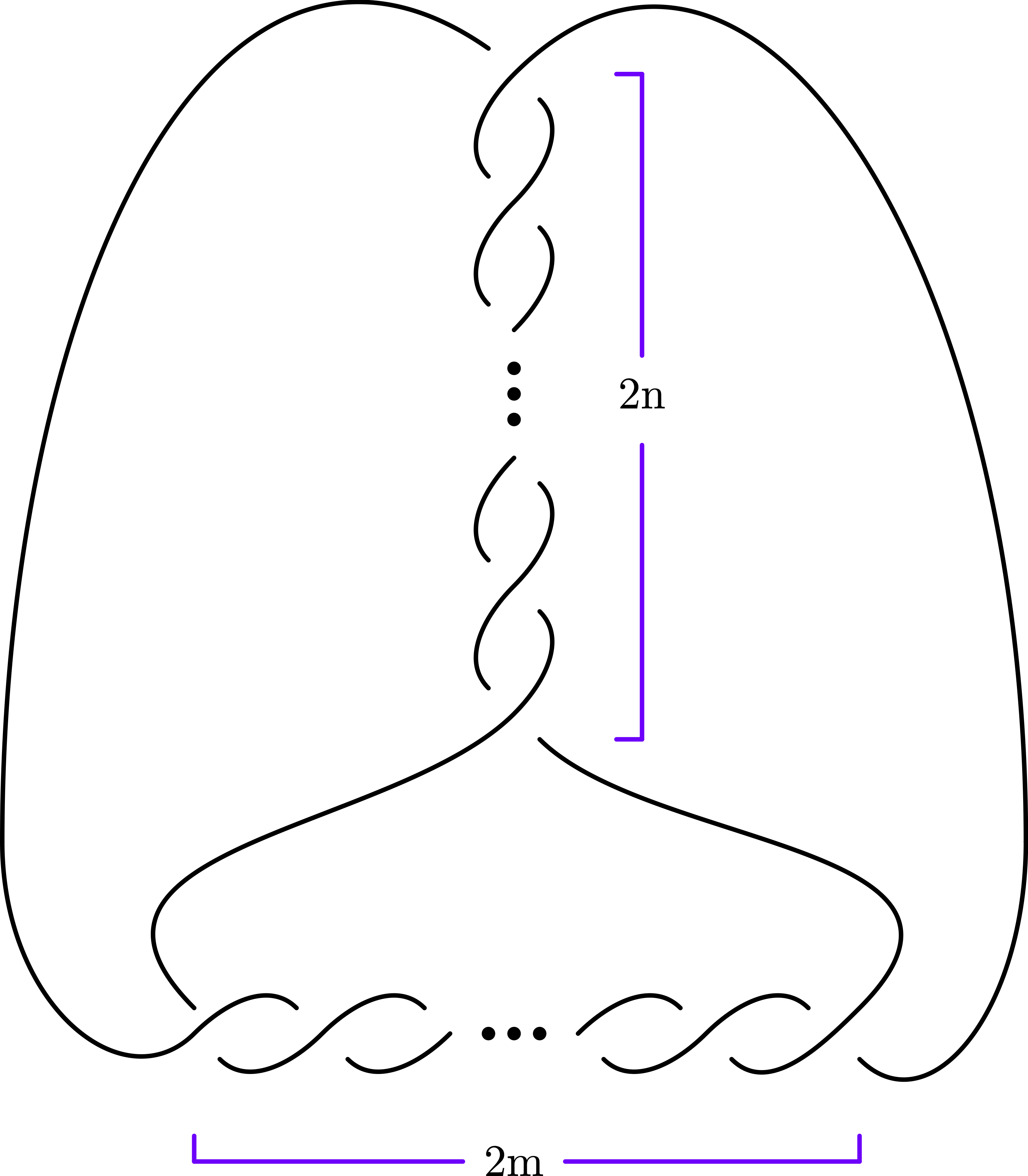}}
\caption{  C(2n,2m)} 
\label{fig:J[2n,2m]}
\end{center}
\end{figure}

A general reference for this section is ~\cite{HS}.
A knot with $2 n$ right-handed vertical crossings and $2m$ left-handed horizontal crossings as in Figure~\ref{fig:J[2n,2m]} is $C(2n, 2m)$ knot according to Conway's notation. 
One can easily check that the slope of $C(2n,2m)$ is $2m/(4nm+1)$ which is equivalent to the knot with slope 
$\left(2n(2m-1)+1\right)/(4nm+1)$~\cite{S1}.  

We will use the following presentation of the fundamental group of $C(2n,2m)$ knot(equivalently, $J(2n,-2m)$ knot) in~\cite{HS}. 
In~\cite{HS}, Hoste and Shanahan asked whether their presentation of the fundamental group for double twist knots can be derived from Schubert’s canonical 2-bridge diagram or not. The following proposition can also be obtained by reading off the fundamental group from the Schubert normal form of $C(2n,2m)$ with slope $2m/(4nm+1)$~\cite{S1,R1} which answers Hoste-Shanahan's question completely for $C(2n,2m)$ knots. 
Let$X_{2n}^{2m}$ be $S^3 \backslash C(2n,2m)$.
\begin{proposition}~\cite[Proposition2.2]{HLMR} ~\cite{R1, S1}\label{theorem:fundamentalGroup}
$$\pi_1(X_{2n}^{2m})=\left\langle s,t \ |\ sw^mt^{-1}w^{-m}=1\right\rangle,$$
where $w=(t^{-1}s)^n(ts^{-1})^n$.
\end{proposition}
\subsection{The Chebychev polynomial}\label{subsec:cheby}
Let $S_k(v)$ be the \emph{Chebychev polynomials} defined by $S_0(v) = 1$, $S_1(v) = v$ and $S_k(v) = vS_{k-1}(v) - S_{k-2}(v)$ for all integers $k$.
The following explicit formula for $S_k(v)$ can be obtained by solving the above recurrence relation~\cite{Tran2}.
\begin{align*}
S_n(v)=\sum_{0 \leq i \leq \lfloor \frac{n}{2} \rfloor} (-1)^i \binom{n-i}{i}v^{n-2i}
\end{align*}
for $n \geq 0$, $S_n(v) = -S_{-n-2}(v)$ for $n \leq -2$, and $S_{-1}(v) = 0$.
The following proposition~\ref{prop:cheby} can be proved using the Cayley-Hamilton theorem~\cite{Tran}.
\begin{proposition}~\cite[Proposition 2.4]{Tran}\label{prop:cheby}
Suppose $V=
\begin{bmatrix}
a & b \\ c & d
\end{bmatrix}
\in \textnormal{SL}_2(\Cb).$
Then
\begin{equation*}
V^k=
\begin{bmatrix}
S_k(v)-d S_{k-1}(v) & b S_{k-1} (v) \\ c S_{k-1} (v) & S_k(v)-a S_{k-1} (v)
\end{bmatrix}
\end{equation*}
where $v=\textnormal{tr}(V)=a+d$.
\end{proposition}
\subsection{The Riley-Mednykh polynomial}
Let
\begin{center}
$$\begin{array}{ccccc}
\rho(s)=\left[\begin{array}{cc}
                       M &       1 \\
                        0      & M^{-1}  
                     \end{array} \right]                          
\text{,} \ \ \
\rho(t)=\left[\begin{array}{cc}
                   M &  0      \\
                   2-v      & M^{-1} 
                 \end{array}  \right],
\end{array}$$
\end{center}
and let
 \begin{center}
$$\begin{array}{cc}
c=\left[\begin{array}{cc}
        0 & -\left(\sqrt{2-v }\right)^{-1}     \\
       \sqrt{2-v} & 0
       \end{array}  \right].
\end{array}$$
\end{center}

Then from the above Proposition~\ref{prop:cheby}, we get the following Theorem~\ref{thm:RMpolynomial}. 
Let $\rho(s)=S$, $\rho(t)=T$ and 
$\rho(w)=W$. 
Then $\textnormal{tr}(T^{-1}S)=v=\textnormal{tr}(TS^{-1})$. Let $v=x+M^2+M^{-2}$. Theorem~\ref{thm:RMpolynomial} can be found in~\cite{Tran3}. We include the proof for readers' convenience.

\begin{theorem}~\cite{Tran3} \label{thm:RMpolynomial}
$\rho$ is a representation of $\pi_1(X_{2n}^{2m})$ if and only if $x$ is a root of the following Riley-Mednykh polynomial,
\begin{align*}
\phi_{2n}^{2m} (x,M)=S_m(z)+\left[-1+x  S_{n-1}(v)\left(S_n(v)+(1-v) S_{n-1}(v)\right)\right]S_{m-1}(z).
\end{align*}
\end{theorem}

\begin{proof}
Since
\begin{align*}
T^{-1}S &=\left[\begin{array}{cc}
                       1 &      M^{-1}  \\
                       M( M^{-2}+(x-2) +M^2)\ \ \ \ \    & M^{-2} +x-1+M^2 
                     \end{array} \right]  \text{ and} \\
T S^{-1} &=\left[\begin{array}{cc}
                   1  &  -M      \\
                   -M^{-1} (M^{-2}+(x-2) +M^2)    & M^{-2} +x-1+M^2 
                 \end{array}  \right],
\end{align*}

\begin{align*}
(T^{-1} S)^n &=\left[\begin{array}{cc}
                       S_n(v)-(v-1) S_{n-1}(v) &      M^{-1} S_{n-1}(v) \\
                        M (v-2) S_{n-1}(v)  \ \ \ \ \  & S_n(v)-S_{n-1}(v)  
                     \end{array} \right] , \\    
(T S^{-1})^n &=\left[\begin{array}{cc}
                       S_n(v)-(v-1) S_{n-1}(v) &      -M S_{n-1}(v) \\
                        -M^{-1} (v-2) S_{n-1}(v)     & S_n(v)-S_{n-1}(v)  
                     \end{array} \right]     
\end{align*}

Hence 
\begin{equation*}
W=(T^{-1} S)^n (T S^{-1})^n 
=\begin{bmatrix}
W_{11} & W_{12} \\ (2-v) W_{12} & W_{22}
\end{bmatrix}
\ \text{where}
\end{equation*}

\begin{align*}
W_{11} & =S^2_n(v)+(2-2v) S_n(v) S_{n-1}(v)+(1+2 M^{-2}-2v-M^{-2}v+v^2) S^2_{n-1}(v) \\
W_{12} & =(M^{-1}-M) S_n(v) S_{n-1}(v)+(M v-M-M^{-1}) S^2_{n-1}(v)\\
W_{22} &=S^2_n(v)-2 S_n(v) S_{n-1}(v)+(1+2 M^2-M^2 v) S^2_{n-1}(v).
\end{align*}

Let $z=\text{Tr}(W)$. Then, since $S_n^2 (v) -vS_n(v)S_{n-1}(v) + S_{n-1}^2(v) = 1$ (by~\cite[Lemma 2.1]{Tran1} or by induction),

\begin{align*}
z=W_{11}+W_{22} &=2(S^2_n(v)-v S_n(v) S_{n-1}(v)+S^2_{n-1}(v))\\
&+(2 M^{-2}+2 M^2-2v-M^{-2}v-M^2 v+v^2)S^2_{n-1}(v)\\
&=2+(v-2) x S^2_{n-1}(v).
\end{align*}

By Proposition~\ref{prop:cheby}, we have
\begin{align*}
(W)^m &=\left[\begin{array}{cc}
                       S_m(z)-W_{22} S_{m-1}(z) &     W_{12}  S_{m-1}(z) \\
                         (2-v) W_{12} S_{m-1}(z)  \ \ \ \ \  & S_m(z)-W_{11} S_{m-1}(z)  
                     \end{array} \right]. 
\end{align*}                       

Therefore $$\textnormal{tr}(SW^mc)/\sqrt{2-v}\left(=(M-M^{-1}) W_{12}  S_{m-1}(z) +S_m(z)-W_{11} S_{m-1}(z)\right)$$ gives 
$\phi_{2n}^{2m} (x,M)$~\cite{HLMR}.
\end{proof}
\section{Longitude}
\label{sec:longitude}
   Let $l = w^m(w^{*})^{m}$, where $w^{*}$ is the word obtained by reversing $w$. Then $l$ is the longitude which is null-homologus in $X_{2n}^{2m}$. Recall $\rho(w)=W$. Let $\widetilde{W}=\rho(w^{*})$. It is easy to see that $\widetilde{W}$ can be written as
   
$$\widetilde{W}=
\begin{bmatrix}
\tilde{W}_{22} & \tilde{W}_{12} \\ (2-v) \tilde{W}_{12} & \tilde{W}_{11}
\end{bmatrix}
$$
where $\tilde{W}_{ij}$ is obtained by $W_{ij}$ by replacing $M$ with $M^{-1}$. Similar computation was introduced in~\cite{HS}. Hence,

\begin{align*}
\tilde{W}_{11} & =S^2_n(v)+(2-2v) S_n(v) S_{n-1}(v)+(1+2 M^{2}-2v-M^{2}v+v^2) S^2_{n-1}(v) \\
\tilde{W}_{12} & =(M-M^{-1}) S_n(v) S_{n-1}(v)+(M^{-1} v-M^{-1}-M) S^2_{n-1}(v)\\
\tilde{W}_{22} &=S^2_n(v)-2 S_n(v) S_{n-1}(v)+(1+2 M^{-2}-M^{-2} v) S^2_{n-1}(v).
\end{align*}

The following lemma was introduced in~\cite{HS} with slightly different coordinates. 

Let $L=\rho(l)_{11}$(the left upper entry of $\rho(l)$). 
\begin{lemma}~\cite{HS}
\label{lem:lemma}
$W_{21} L+\tilde{W}_{21}=0$.
\end{lemma}

\begin{theorem}
\label{thm:longitude}
\begin{align*}
L=-\frac{M^2 (S_n(v)-S_{n-1}(v))-(S_{n-1}(v)-S_{n-2}(v))}{(S_n(v)-S_{n-1}(v))-M^2 (S_{n-1}(v)-S_{n-2}(v))}.
\end{align*}
\end{theorem}

\begin{proof}
By directly computing $W_{21} L+\tilde{W}_{21}=0$ in Lemma~\ref{lem:lemma} and substituting $S_n(v)+S_{n-2}(v)$ for $v S_{n-1}(v)$, the theorem follows.
\end{proof}
\section{Schl\"{a}fli formula for the generalized Chern-Simons function}
\label{sec:CSfunction}

The general references for this section are~\cite{HLM3,HLM2,Y1,MeyRub1,HL,HL2} and~\cite{HLMR}. 

In~\cite{HLM3}, Hilden, Lozano, and Montesinos-Amilibia defined the \emph{generalized Chern-Simons function} on the oriented cone-manifold structures which matches up with the Chern-Simons invariant when the cone-manifold is the Riemannian manifold.

Below, we briefly introduce the generalized Chern-Simons function on the family of $C(2n,2m)$ cone-manifold structures. For an oriented knot $C(2n,2m)$, we orient its chosen meridian $s$ such that the orientation of $s$ followed by the orientation of $C(2n,2m)$ coincides with orientation of $\mathbb{S}^3$. Here, we use the definition of the lens space in~\cite{HLM2} so that we can have the right orientation when it is combined with the following frame field. On the Riemannian manifold $\mathbb{S}^3-C(2n,2m)-s$ we choose a \textit{special} frame field 
$\Gamma=(e_1,e_2,e_3)$ which is an orthonormal frame field such that for each point $x$ near $C(2n,2m)$, $e_1(x)$ has the direction given by knot's orientation, $e_2(x)$ has the tangent direction of the meridian curve, and $e_3(x)$ has the knot to point $x$ direction.  Such a special frame field always exists by Proposition $3.1$ of~\cite{HLM3}. From $\Gamma$ we obtain an orthonormal frame field 
$\Gamma_{\alpha}$ on $X_{2n}^{2m} \!  (\alpha)-s$ by the Gram-Schmidt orthogonalization process with respect to the Riemannian structure of the cone manifold $X_{2n}^{2m} \!  (\alpha)$. Moreover, it can be made special by deforming it in a neighbourhood of the singular set $C(2n,2m)$ and $s$, if necessary. Thus, $\Gamma^{\prime}$ is an extension of $\Gamma$ to $S^3-C(2n,2m)$. To the cone-manifold $X_{2n}^{2m} \!  (\alpha)$, we assign a real number

\begin{equation*}
I\left(X_{2n}^{2m} \!  (\alpha)\right)=\frac{1}{2} \int_{\Gamma(S^3-T_{2n}-s)}Q-\frac{1}{4 \pi} \tau(s,\Gamma^{\prime})-\frac{1}{4 \pi} \left(\frac{\beta \alpha}{2 \pi}\right),
\end{equation*}

\noindent where $-2 \pi \leq \beta \leq 2 \pi$ is the angle of the lifted holonomy of the singular locus of $X_{2n}^{2m} \!  (\alpha)$, $Q$ is the Chern-Simons form:

\begin{equation*}
Q=\frac{1}{4 \pi^2} \left(\theta_{12} \wedge \theta_{13} \wedge \theta_{23} + \theta_{12} \wedge \Omega_{12} + \theta_{13} \wedge \Omega_{13} + \theta_{23} \wedge \Omega_{23} \right),
\end{equation*}

\noindent and 

\begin{equation*}
\tau(s,\Gamma^{\prime})=-\int_{\Gamma^{\prime}(s)} \theta_{23},
\end{equation*}

\noindent where ($\theta_{ij}$) is the connection $1$-form, ($\Omega_{ij}$) is the curvature $2$-form of the Riemannian connection on $X_{2n}^{2m} \!  (\alpha)$ and the integral is over the orthonormalizations of the same frame field. When $\alpha = \frac{2 \pi}{k}$ for some positive integer, 
$I \left(X_{2n}^{2m} \! \! \left(\frac{2 \pi}{k}\right)\right)$ (mod $\frac{1}{k}$ if $k$ is even or mod $\frac{1}{2k}$ if $k$ is odd) is independent of the frame field $\Gamma$ and of the representative in the equivalence class $\overline{\beta}$ and hence becomes an invariant of the orbifold $X_{2n}^{2m} \! \! \left(\frac{2 \pi}{k}\right)$. The quantity $I \left(X_{2n}^{2m} \! \! \left(\frac{2 \pi}{k}\right)\right)$ (mod $\frac{1}{k}$ if $k$ is even or mod $\frac{1}{2k}$ if $k$ is odd) is called \emph{the Chern-Simons invariant of the orbifold} and is denoted by 
$\text{\textnormal{cs}} \left(X_{2n}^{2m} \! \! \left(\frac{2 \pi}{k}\right) \right)$.

We have the following ``Schl\"{a}fli formula'' for the generalized Chern-Simons function on the family of $C(2n,2m)$ cone-manifold structures.

\begin{theorem}(~\cite[Theorem 1.2]{HLM2})~\label{theorem:schlafli}
For a family of geometric cone-manifold structures, $X_{2n}^{2m} \!  (\alpha)$, and differentiable functions $\alpha(t)$ and $\beta(t)$ of $t$ we have
\begin{equation*}
dI \left(X_{2n}^{2m} \!  (\alpha)\right)=-\frac{1}{4 \pi^2} \beta d \alpha.
\end{equation*} 
\end{theorem}

\section{Proof of the theorem~\ref{theorem:main}} \label{sec:proof}

For $n \geq 1$ and $M=e^{i \frac{\alpha}{2}}$, $\phi_{2n}^{2m}(x,M)$ have $2mn$ component zeros. The component which passes through $$(x_1,x_2)=\left(2-2 \cos{\left(\frac{\pi(2m+1)}{4nm+1}\right)},2-2 \cos{\left(\frac{\pi(2m-1)}{4nm+1}\right)}\right)$$ at $\alpha=\pi$ is the geometric component by ~\cite[Theorem 2.1]{HLM2}. Note that $2-x_1>0$ and $2-x_2>0$. For each $C(2n,2m)$, there exists an angle $\alpha_0 \in [\frac{2\pi}{3},\pi)$ such that $C(2n,2m)$ is hyperbolic for $\alpha \in (0, \alpha_0)$, Euclidean for 
$\alpha=\alpha_0$, and spherical for $\alpha \in (\alpha_0, \pi]$ \cite{P2,HLM1,K1,PW}.
Denote by $D(X_{2n}^{2m} \!  (\alpha))$ be the set of  zeros of the discriminant of 
$\phi_{2n}^{2m}(x,e^{i \frac{\alpha}{2}})$ over $x$. Then $\alpha_0$ will be one of $D(X_{2n}^{2m} \!  (\alpha))$. 

 On the geometric component we can calculate
 the Chern-Simons invariant of an orbifold 
$X_{2n}^{2m}(\frac{2 \pi}{k})$ (mod $\frac{1}{k}$ if $k$ is even or mod $\frac{1}{2k}$ if $k$ is odd), where $k$ is a positive integer such that $k$-fold cyclic covering of $X_{2n}^{2m}(\frac{2 \pi}{k})$ is hyperbolic:
\begin{align*}
& \text{\textnormal{cs}}\left(X_{2n}^{2m} \left(\frac{2 \pi}{k} \right)\right) 
                        \equiv I \left(X_{2n}^{2m} \left(\frac{2 \pi}{k} \right)\right) 
                      \ \ \ \ \ \ \ \ \ \ \ \  \left(\text{mod} \ \frac{1}{k}\right) \\
                        & \equiv I \left(X_{2n}^{2m}( \pi) \right)
                          +\frac{1}{4 \pi^2}\int_{\frac{2 \pi}{k}}^{\pi} \beta \: d\alpha 
                      \ \ \ \ \ \ \ \ \ \ \ \ \  \ \ \ \ \ \ \  \left(\text{mod} \ \frac{1}{k}\right) \\
                        & \equiv \frac{1}{2} \text{\textnormal{cs}}\left(L(4nm+1,2n(2m-1)+1) \right) \\
                        &+\frac{1}{4 \pi^2}\int_{\frac{2 \pi}{k}}^{\alpha_0} Im \left(2*\log \left(-\frac{M^2 (S_n(v)-S_{n-1}(v))-(S_{n-1}(v)-S_{n-2}(v))}{(S_n(v)-S_{n-1}(v))-M^2 (S_{n-1}(v)-S_{n-2}(v))}\right)\right) \: d\alpha \\
& +\frac{1}{4 \pi^2}\int_{\alpha_0}^{\pi}
 Im \left(\log \left(-\frac{M^2 (S_n(v_1)-S_{n-1}(v_1))-(S_{n-1}(v_1)-S_{n-2}(v_1))}{(S_n(v_1)-S_{n-1}(v_1))-M^2 (S_{n-1}(v_1)-S_{n-2}(v_1))}\right) \right) \: d\alpha \\
& +\frac{1}{4 \pi^2}\int_{\alpha_0}^{\pi}
 Im \left(\log \left(-\frac{M^2 (S_n(v_2)-S_{n-1}(v_2))-(S_{n-1}(v_2)-S_{n-2}(v_2))}{(S_n(v_2)-S_{n-1}(v_2))-M^2 (S_{n-1}(v_2)-S_{n-2}(v_2))}\right)\right) \: d\alpha\\
& \left( \text{mod} \ \frac{1}{k}\ \text{if $k$ is even or }  \text{mod} \ \frac{1}{2k}\ \text{if $k$ is odd} \right)
\end{align*}
where the second equivalence comes from Theorem~\ref{theorem:schlafli} and the third equivalence comes from the fact that 
$I \left(X_{2n}^{2m}(\pi)\right) \equiv \frac{1}{2} \text{\textnormal{cs}}\left(L(4nm+1,2n(2m-1)+1) \right)$  
$\left(\text{mod }\frac{1}{2}\right)$, Theorem~\ref{thm:longitude}, and geometric interpretations of hyperbolic and spherical holonomy representations.

The following theorem gives the Chern-Simons invariant of the Lens space $L(4nm+1,2n(2m-1)+1)$.

\begin{theorem}(~\cite[Theorem 1.3]{HLM2}) \label{theorem:Lens}
\begin{align*}
\text{\textnormal{cs}} \left(L \left(4nm+1,2n(2m-1)+1\right)\right) \equiv \frac{m-n}{4n m+1} && (\text{mod}\ 1).
\end{align*}
\end{theorem}

\section{Chern-Simons invariants of the hyperbolic $C(2n,2m)$ knot orbifolds and of its cyclic coverings}

The table~\ref{tab1} gives the approximate Chern-Simons invariant of 
$C(2n,2m)$ for $n$ between $1$ and $4$, $m$ between $1$ and $4$ with $n \geq m$. Since $C(2,2)$, $C(4,4)$, $C(6,6)$, $C(8,8)$ are amphicheiral knots, their Chern-Simons invariants are $zero$ as expected.  We used Simpson's rule for the approximation with $2 \times 10^4$ ($10^4$ in Simpson's rule) intervals from $0$ to $\alpha_0$ and $2 \times10^4$ ($10^4$ in Simpson's rule) intervals from $\alpha_0$ to $\pi$. 
The table~\ref{table1-1} gives the approximate Chern-Simons invariant of the hyperbolic orbifold, 
$\text{\textnormal{cs}} \left(X_{2n}^{2m} (\frac{2 \pi}{k})\right)$ for $n$ between $1$ and $4$, $m$ between $1$ and $4$ with 
$n \geq m$, and for $k$ between $3$ and $10$, and of its cyclic covering, $\text{\textnormal{cs}} \left(M_k (X_{2n}^{2m})\right)$   except amphicheiral knots. We used Simpson's rule for the approximation with $2 \times10^2$ ($10^2$ in Simpson's rule) intervals from $2 \pi/k$ to $\alpha_0$ and $2 \times 10^2$ ($10^2$ in Simpson's rule) intervals from $\alpha_0$ to $\pi$. 

We used Mathematica for the calculations. We record here that our data in Table~\ref{tab1} and those obtained from  SnapPy~\cite{SnapPy} match up up to existing decimal points and our data in Table~\ref{table1-1}. For computational reasons, we need $\alpha_0$. $\alpha_0$ is the bifercation point of the geometric solution of the Riley-Mednykh polynomial as described in Theorem~\ref{theorem:main}.

\begin{table}[h] 
\begin{tabular}{|c|c|c|c|}
\hline
2n & 2m & $\alpha_0$ & $\text{\textnormal{cs}}\left(X_{2n}^{2m}\right)$  \\
\hline
2 & 2 & 2.094395102393195 & 0\text{} \\
4 & 2 & 2.574140778131840 & 0.34402298 \text{} \\
6 & 2 & 2.750685152010280 & 0.27786688 \text{} \\
8 & 2 & 2.843209532683532 & 0.24222232 \text{} \\
4 & 4 & 2.847642272262783 & 0 \text{} \\
6 & 4 & 2.942465754372979 & 0.42782933\text{} \\
8 & 4 & 2.990939179603150 & 0.38923730\text{} \\
6 & 6 & 3.007517657179940 & 0 \text{} \\
8 & 6 & 3.040474611156828 &  0.46103929\text{} \\
8 & 8 & 3.065453796328835 & 0 \text{} \\
\hline
\end{tabular}
\bigskip
\caption{Chern-Simons invariant of $X_{2n}^{2m}$ for $n$ between $1$ and $4$, $m$ between $1$ and $4$ with $n \geq m$ except amphicheiral knots.}
\label{tab1}
\end{table}


\begin{table}
\begin{tabular}{ll}
\begin{tabular}{|c|c|c|}
\hline
 $k$ & $\text{\textnormal{cs}} \left(X_4^2(\frac{2 \pi}{k})\right)$ & $\text{\textnormal{cs}} \left(M_k( X_4^2)\right)$ \\
\hline
 3 & 0.0875301 \text{} & 0.26259 \text{} \\
 4 & 0.144925 & 0.579699 \\
 5 & 0.0784576 \text{} & 0.392288 \text{} \\
 6 & 0.0351571 & 0.210943 \\
 7 & 0.00506505 \text{} & 0.0354553 \text{} \\
 8 & 0.108039 & 0.864313 \\
 9 & 0.0218112 \text{} & 0.196301 \text{} \\
 10 & 0.0530574 & 0.530574 \\
\hline
\end{tabular}
&
\begin{tabular}{|c|c|c|}
\hline
 $k$ &  $\text{\textnormal{cs}} \left(X_6^2(\frac{2 \pi}{k})\right)$ & $\text{\textnormal{cs}} \left(M_k (X_6^2)\right)$ \\
\hline
 3 & 0.0449535 \text{} & 0.13486 \text{} \\
 4 & 0.0876043 & 0.350417 \\
 5 & 0.0165337 \text{} & 0.0826684 \text{} \\
 6 & 0.138167 & 0.829004 \\
 7 & 0.0120078 \text{} & 0.0840545 \text{} \\
 8 & 0.0430876 & 0.3447 \\
 9 & 0.012125 \text{} & 0.109125 \text{} \\
 10 & 0.0876213 & 0.876213 \\
\hline
\end{tabular}
\end{tabular}

\bigskip

\begin{tabular}{ll}
\begin{tabular}{|c|c|c|}
\hline
 $k$ &   $\text{\textnormal{cs}} \left(X_8^2(\frac{2 \pi}{k})\right)$ & $\text{\textnormal{cs}} \left(M_k (X_8^2)\right)$ \\
\hline
 3 & 0.0161266 \text{} & 0.0483799 \text{} \\
 4 & 0.0536832 & 0.214733 \\
 5 & 0.0817026 \text{} & 0.408513 \text{} \\
 6 & 0.103012 & 0.618074 \\
 7 & 0.0481239 \text{} & 0.336867 \text{} \\
 8 & 0.00768503 & 0.0614802 \\
 9 & 0.032221 \text{} & 0.289989 \text{} \\
 10 & 0.0521232 & 0.521232 \\
\hline
\end{tabular}
&
\begin{tabular}{|c|c|c|}
\hline
 $k$ &  $\text{\textnormal{cs}} \left(X_6^4 (\frac{2 \pi}{k})\right)$ & $\text{\textnormal{cs}} \left( M_k (X_6^4)\right)$ \\
\hline
3 & 0.125912 \text{} & 0.377736 \text{} \\
 4 & 0.192764 \text{} & 0.771058 \text{} \\
 5 & 0.0360431 \text{} & 0.180216 \text{} \\
 6 & 0.0996796 \text{} & 0.598077 \text{} \\
 7 & 0.00284328 \text{} & 0.0199029 \text{} \\
 8 & 0.0554674 \text{} & 0.443739 \text{} \\
 9 & 0.0409685 \text{} & 0.368717 \text{} \\
 10 & 0.0294401 \text{} & 0.294401 \text{} \\
\hline
\end{tabular}
\end{tabular}

\bigskip

\begin{tabular}{ll}
\begin{tabular}{|c|c|c|}
\hline
 $k$ &  $\text{\textnormal{cs}} \left(X_8^4(\frac{2 \pi}{k})\right)$ & $\text{\textnormal{cs}} \left(M_k (X_8^4)\right)$ \\
\hline
 3 & 0.098074 \text{} & 0.294222 \text{} \\
 4 & 0.157843 \text{} & 0.631371 \text{} \\
 5 & 0.0993608 \text{} & 0.496804 \text{} \\
 6 & 0.0622858 \text{} & 0.373715 \text{} \\
 7 & 0.0365103 \text{} & 0.255572 \text{} \\
 8 & 0.0174882 \text{} & 0.139906 \text{} \\
 9 & 0.00284881 \text{} & 0.0256393 \text{} \\
 10 & 0.091224 \text{} & 0.91224 \text{} \\
\hline
\end{tabular}
&
\begin{tabular}{|c|c|c|}
\hline
 $k$ &  $\text{\textnormal{cs}} \left(X_8^6 (\frac{2 \pi}{k})\right)$ & $\text{\textnormal{cs}} \left(M_k (X_8^6)\right)$ \\
\hline
3 & 0.138854 \text{} & 0.416562 \text{} \\
 4 & 0.214725 \text{} & 0.858898 \text{} \\
 5 & 0.0628859 \text{} & 0.31443 \text{} \\
 6 & 0.128841 \text{} & 0.773046 \text{} \\
 7 & 0.0332457 \text{} & 0.23272 \text{} \\
 8 & 0.0866094 \text{} & 0.692875 \text{} \\
 9 & 0.0170324 \text{} & 0.153291 \text{} \\
10 & 0.0613865 \text{} & 0.613865 \text{} \\
\hline
\end{tabular}
\end{tabular}

\bigskip

\caption{Chern-Simons invariant of the hyperbolic orbifold, $\text{\textnormal{cs}} \left(X_{2n}^{2m} (\frac{2 \pi}{k})\right)\left( \text{mod} \ \frac{1}{k}\ \text{if $k$ is even or }  \text{mod} \ \frac{1}{2k}\ \text{if $k$ is odd} \right)$ for $n$ between $1$ and $4$, $m$ between $1$ and $4$ with $n \geq m$, and for $k$ between $3$ and $10$, and of its cyclic covering, 
$\text{\textnormal{cs}} \left(M_k (X_{2n}^{2m})\right)$  except amphicheiral knots.}
\label{table1-1}
\end{table}

Note that the Chern-Simons invariant of the hyperbolic orbifold  $\text{\textnormal{cs}} \left(X_{2n}^{2m} (\frac{2 \pi}{k})\right)$ is only defined modulo $\frac{1}{k}$~\cite[Theorem 1.4]{HLM2} and  we only get modulo $\frac{1}{2k}$ for $k$ odd~\cite[Theorem 1.4]{HLM2}.


\bibliographystyle{pigc_plain}
\bibliography{conereference}

\printArticleAuthorsInfo{\thearticlesnum}

\end{document}